\documentclass[11pt]{article}
\usepackage{graphicx}
\usepackage{amssymb}
\usepackage{epstopdf}
\usepackage{amsmath,amsthm,amscd,amssymb}
\usepackage{latexsym}
\usepackage[colorlinks,citecolor=red,pagebackref,hypertexnames=false]{hyperref}
\usepackage{geometry}  
\usepackage{cite}              
\numberwithin{equation}{section}

\theoremstyle{plain}
\newtheorem{theorem}{Theorem}[section]
\newtheorem{lemma}[theorem]{Lemma}

\theoremstyle{definition}

\newtheorem{case[theorem]}{Case}

\theoremstyle{remark}

\numberwithin{equation}{section}

\begin{document}

\title{Hamming Distances in Vector Spaces over Finite Fields}

\author{Esen Aksoy Yazici}

\maketitle

\begin{abstract}
Let $\mathbb{F}_q$ be the finite field of order $q$ and $E\subset \mathbb{F}_q^d$, where $4|d$.
Using Fourier analytic techniques, we prove that  if $|E|>\frac{q^{d-1}}{d}\binom{d}{d/2}\binom{d/2}{d/4}$, then the points of $E$ determine a Hamming distance $r$ for every even $r$.

\end{abstract}

\section{Introduction}

Let $\mathbb{F}_q$ be the finite field with order $q$, where $q=p^l$ and $p$ is an odd prime. In the vector space $\mathbb{F}_q^d$, we can consider the following distance map

\begin{equation}\label{distance}
 \lambda: (x,y)\longmapsto\|x-y\|=(x_1-y_1)^2+\ldots+(x_d-y_d)^2.
 \end{equation}

For  $E\subset \mathbb{F}_q^d$, let $\Delta(E)$ denote the set of distances determined by the points of $E$ that is, $$\Delta(E):=\{\|x-y\|: x,y\in E\}.$$

The Erd\H{o}s-Falconer distance problem in $\mathbb{F}_q^d$ asks for a threshold on the size  $E\subset \mathbb{F}_q^d$ so that $\Delta(E)$ contains a positive proportion of $\mathbb{F}_q.$

In \cite{IR07},  Iosevich and Rudnev proved that for $E\subset \mathbb{F}_q^d$ if $|E|>cq^{\frac{d+1}{2}}$ for a sufficiently large constant $c$, then $\Delta(E)=\mathbb{F}_q$. 

Erd\H{o}s-Falconer distance problem in modules $\mathbb{Z}_q^d$ over the cyclic rings $\mathbb{Z}_q$ was studied by Covert, Iosevich and Pakianathan in \cite{CIP}. More precisely, it is proven that for $E\subset \mathbb{Z}_q^d$, where $q=p^l$, if $|E|\gg l(l+1)q^{\frac{(2l-1)d}{2l}+\frac{1}{2l}}$, then $\Delta(E)$ contains all unit elements of $\mathbb{Z}_q$. 

For more literature on the distance introduced in (\ref{distance}) and related geometric configurations, we refer to \cite{BHIPR13, BIP, CEHIK12, CHISU, HI, HIKR11, IRZ12} and the references therein.

Here, in this paper, we tackle a similar problem related to coding theory. Instead of the distance given in (\ref{distance}), we consider the Hamming distance in $\mathbb{F}_q^d$, a key notion in coding theory, and ask similar geometric configurations in $\mathbb{F}_q^d$. We note that the approach we use to prove the main theorem of this paper is analogous to the one employed in \cite{CIP} and \cite{IR07}. Let us first recall the necessary notion. 

For two vectors $x=(x_1,\dots ,x_d),y=(y_1,\dots, y_d)\in \mathbb{F}_q^d$,  the Hamming distance between $x$ and $y$ is defined as 
$$|x-y|=\sum_{i=1}^{d}d(x_i,y_i)$$
where

\begin{equation*}
d(x_i,y_i)=1-\delta_{x_i,y_i}= \left\{
\begin{array}{rl}
0 & \text{if } x_i=y_i,\\
1 & \text{if } x _i\ne y_i.
\end{array} \right.
\end{equation*}

In other words, the Hamming distance $|x-y|$ between $x$ and $y$ is the number of coordinates in which $x$ and $y$ differ. In particular, $|x|$  is the number of nonzero coordinates of $x$. We will denote the Hamming weight of $x$ as $wt(x)$.

The question we will be dealing with in this note is that for subsets $E$ of $\mathbb{F}_q^d$, which can be seen as a code over $ \mathbb{F}_q$, how large does the size of $E$ need to be to guarantee that $E$ contains the desired set of Hamming distances.

\subsection{Main Result}

\begin{theorem}\label{main}
Let $E\subset \mathbb{F}_q^d$ where $4|d$.   If $|E|>\frac{q^{d-1}}{d}\binom{d}{d/2}\binom{d/2}{d/4}$, then the points of $E$ determine a Hamming distance $r$ for every even $r$.

\end{theorem}

\subsection{Fourier Analysis in $\mathbb{F}_q^d$}

 Let $ f : \mathbb {F}_{q}^d\to \mathbb{C}$. The Fourier transform of $f$ is defined as 
$$\widehat {f}(m)=q^{-d}\sum_{x\in \mathbb {F}_{q}^d}\chi(-x\cdot m)f(x),$$
where $\chi(z)=e^{\frac{2\pi i Tr{(z)}}{q}}$, $q=p^l$, p prime, and $Tr: \mathbb{F}_q\to \mathbb{F}_p$ is the Galois trace.

We recall the following properties of Fourier transform.
\begin{equation*} 
  q^{-d}\sum_{x\in \mathbb {F}_{q}^d}\chi(x\cdot m) = \left\{
    \begin{array}{rl}
      1, & \text{if } m=0\\
      0, & otherwise
      \end{array} \right.  \qquad \text{(Orthogonality)}
\end{equation*}

\begin{equation*}
f(x)=\sum_{m\in \mathbb {F}_{q}^d}\chi(x\cdot m)\widehat{f}(m) \qquad \text{(Inversion)}
\end{equation*}
\begin{equation*}
\sum_{m\in \mathbb {F}_{q}^d}|\widehat{f}(m)|^2=q^{-d}\sum_{x\in \mathbb {F}_{q}^d}|f(x)|^2. \qquad \text{(Plancherel)}
\end{equation*}

\vskip.125in 

\section{Proof of Main Result} For the proof of Theorem \ref{main}, we make use of the following lemmas:

\begin{lemma} Let $S_r(u)=\{v\in \mathbb{F}_q^d: |u-v|=r\}$ be the sphere of radius $r$ centered at $u\in \mathbb{F}_q^d$. Then
$$|S_r(u)|=(q-1)^r\binom{d}{r}.$$
\end{lemma}

\begin{proof} If $v\in S_r(u)$, then $u$ and $v$ differ in $r$ coordinates. Note that we have $\binom{d}{r}$ ways of choosing those $r$ coordinates, and for each of these $r$ coordinates of $v$ we have $q-1$ choices. 
\end{proof}

\begin{lemma}\label{supshat} Let $S_r:=S_r(0)=\{v\in \mathbb{F}_q^d: |v|=r\}$ denote the sphere of radius $r$ centered at $0\in \mathbb{F}_q^d$, where $4|d$ , and identify $S_r$ with its indicator function.Then
\begin{eqnarray*}
\sup_{0\ne m\in  \mathbb{F}_q^d} |\widehat{S}_r(m)|&= &q^{-d}\sup_{0\ne m\in \mathbb{F}_q^d} |K_{r}(wt(m))|\\
&\le& 
\begin{cases}
   q^{-d} \binom{d}{d/2} \binom{d/2}{d/4}      & \text{if } wt(m) \text{is even}  \\
   q^{-d}(q-1)^{r-1}\frac{\binom{d}{r}}{d}\binom{d}{d/2}\binom{d/2}{d/4}        & \text{if } wt(m) \text{is odd and}\; r\; \text{is even}
  \end{cases}
\end{eqnarray*}

\end{lemma}

\begin{proof}
\begin{eqnarray*}
\widehat{S}_r(m)&=&q^{-d}\sum_{x\in \mathbb{F}_q^d}\chi(-x\cdot m)S_r(x)\\
&=&q^{-d}\sum_{\substack{x_{i_1},\dots,x_{i_r}\in \mathbb{F}_q^{*}\\i_j\in \{1,\dots d\}\\ i_j\ne i_k } } \chi (-x_{i_1}m_{i_1}-\dots-x_{i_r}m_{i_r})\\
&=&q^{-d}\sum_{\substack{x_{i_1},\dots,x_{i_r}\in \mathbb{F}_q^{*}\\i_j\in \{1,\dots d\}\\ i_j\ne i_k } }e^{\frac{-2\pi i}{q}(x_{i_1}m_{i_1}+\dots +x_{i_r}m_{i_r}) }\\
&=&q^{-d}\sum_{\substack{|\mathcal{I}^k|=r\\\mathcal{I}^k=(k_1,\dots,k_r)}} \prod_{i=1}^{r}\sum_{x_i\in \mathbb{F}_q^{*}} e^{\frac{-2\pi i}{q}(x_im_{k_i})}\\
&=&q^{-d}\sum_{\substack{\{k_1,\dots,k_r\}\subset \{1,\dots,d\}\\k_i<k_j\; \text{for}\; i<j }}\left(\sum_{x_1\in \mathbb{F}_q^*}e^{\frac{-2\pi i}{q}(x_1m_{k_1})} \dots\sum_{x_r\in \mathbb{F}_q^{*}}e^{\frac{-2\pi i}{q}(x_rm_{k_r})} \right)
\end{eqnarray*}
\end{proof}

First note that 

\begin{equation*}
\sum_{x_i\in \mathbb{F}_q^*}e^{\frac{-2\pi i}{q}(x_im_{k_i})} =\left\{
\begin{array}{rl}
q-1& \text{if } m_{k_i}=0\\ 
-1& \text{if } m_{k_i}\ne0.
\end{array} \right.
\end{equation*}

Now let $wt(m)=t$, $m=(m_1,\dots,m_t,\dots,m_d)$, where $m_i\ne0$ for $i=1,\dots, t$ and $m_i=0$ for $i=t+1,\dots,d$.  For a fixed $\mathcal{I}^k=(k_1,\dots,k_r)$,
let 

$$S_{\mathcal{I}^k}=\left(\sum_{x_1\in \mathbb{F}_q^*}e^{\frac{-2\pi i}{q}(x_1m_{k_1})}  \dots \sum_{x_r\in \mathbb{F}_q^{*}}e^{\frac{-2\pi i}{q}(x_rm_{k_r})} \right) .$$ Here  if $i$ coordinates  of $(m_{k_1}\dots,m_t,\dots,m_{k_r})$ are nonzero, then we get
$$S_{\mathcal{I}^k}=(-1)^{i}(q-1)^{r-i},$$
and we have $\binom{t}{i}\binom{d-t}{r-i}$ many such $(m_{k_1}\dots,m_t,\dots,m_{k_r})$. Summing over all possible $i$'s, $i=0,\dots,t$, we get

\begin{eqnarray}\label{shat}
\widehat{S}_r(m)&=&q^{-d}\sum_{i=0}^{r}\binom{t}{i}\binom{d-t}{r-i}(-1)^{i}(q-1)^{r-i}\\
&=&q^{-d}K_r(t)=q^{-d}K_r(wt(m))\nonumber
\end{eqnarray}
where $K_r(\cdot)$ denotes the Krawtchouk  polynomial. 

We will make use of the following two lemmas from \cite{KL}.

\begin{lemma}\emph{\cite[Lemma 1]{KL}}\label{first} For $d$ and $i$ even
$$|K_{k}(i)|\le |K_{d/2}(i)|$$
\end{lemma}

\begin{lemma}\emph{\cite[Lemma 2]{KL}} \label{second} For $k$ integer, $d$ and $i$ even
$$|K_i(k)|\le \frac{\binom{d}{d/2} \binom{d/2}{i/2} }{\binom{d}{k}}$$ 
\end{lemma}

Now using Lemmas \ref{first} and \ref{second}, we immediately obtain that if $wt(m)$ is even, then 

$$\sup |\widehat{S}_r(m) | \le q^{-d} \binom{d}{d/2} \binom {d/2}{d/4} $$

On the other hand, if $wt(m)=i$ is odd, then using the symmetry relation of Krowchouk polynomials, now assuming that $r$ is even , we obtain 

\begin{eqnarray*}
 |\widehat{S}_r(m) |&=&q^{-d}K_r(wt(m))\\
 &=&q^{-d}K_{r}(i)\\
 &=&q^{-d}\frac{(q-1)^r\binom{d}{r}K_i(r) }{(q-1)^{i}\binom{d}{i}}\\
 &\le& q^{-d}(q-1)^{r-i} \frac{\binom{d}{r}}{ \binom{d}{i}} \binom{d}{d/2} \binom {d/2}{d/4}\\
 &\le& q^{-d}(q-1)^{r-1}\frac{ \binom{d}{r} }{d} \binom{d}{d/2} \binom {d/2}{d/4}
\end{eqnarray*}

\begin{proof}[Proof of Theorem \ref{main}]
Let $0<r< d$  be even. Let  $\lambda_r=|\{(x,y)\in E\times E: |x-y|=r\}|$. Then
\begin{eqnarray}\label{lambda}
\lambda_r&=&\sum_{x,y\in \mathbb{F}_q^d}E(x)E(y)S_r(x-y)\nonumber \\
&=&\sum_{x,y,m\in \mathbb{F}_q^d}E(x)E(y)\widehat{S}_r(m)\chi(m\cdot(x-y))\nonumber\\
&=&q^{2d}\sum_{m}|\widehat{E}(m)|^2\widehat{S}_r(m)\nonumber\\
&=&q^{2d}|\widehat{E}(0)|^2\widehat{S}_r(0)+q^{2d}\sum_{m\ne 0}|\widehat{E}(m)|^2\widehat{S}_r(m)\nonumber\\
&=&q^{-d}|E|^2|S_r|+q^{2d}\sum_{m\ne 0}|\widehat{E}(m)|^2\widehat{S}_r(m)\nonumber\\
&=&q^{-d}|E|^2(q-1)^r\binom{d}{r}+q^{2d}\sum_{\substack{m\ne 0\\wt(m) \text{ is even}}}|\widehat{E}(m)|^2\widehat{S}_r(m)+q^{2d}\sum_{\substack{m\ne 0\\wt(m) \text{ is odd}}}|\widehat{E}(m)|^2\widehat{S}_r(m)\nonumber\\
&=&q^{-d}|E|^2(q-1)^r\binom{d}{r}+I+II
\end{eqnarray}
where 
$$ I=q^{2d}\sum_{\substack{m\ne 0\\wt(m) \text{ is even}}}|\widehat{E}(m)|^2\widehat{S}_r(m) $$
and
$$ II=q^{2d}\sum_{\substack{m\ne 0\\wt(m) \text{ is odd}}}|\widehat{E}(m)|^2\widehat{S}_r(m) $$

 We will  first estimate $| I |$. By Lemma \ref{supshat} and Plancherel identity, it follows that
\begin{eqnarray*}
|I|&\leq& q^{2d} q^{-d}    \binom{d}{d/2} \binom{d/2}{d/4} \sum_{\substack{m\ne 0\\wt(m) \text{ is even}}}|\widehat{E}(m)|^2\\
&\le& q^{d}    \binom{d}{d/2} \binom{d/2}{d/4}\sum_{m}|\widehat{E}(m)|^2\\
&\leq& q^{d}    \binom{d}{d/2} \binom{d/2}{d/4}q^{-d}|E|\\
&=&  \binom{d}{d/2} \binom{d/2}{d/4}|E|
\end{eqnarray*} 
Now we will estimate $|II|$. Again using Lemma \ref{supshat} and Plancherel identity, we obtain that
\begin{eqnarray*}
|II|&\le& q^{2d}q^{-d}(q-1)^{r-1}\frac{\binom{d}{r}}{d}\binom{d}{d/2}\binom{d/2}{d/4}  \sum_{\substack{m\ne 0\\wt(m) \text{ is odd}}}|\widehat{E}(m)|^2\\
&\le& q^{d}(q-1)^{r-1}\frac{\binom{d}{r}}{d}\binom{d}{d/2}\binom{d/2}{d/4}  \sum_{m}|\widehat{E}(m)|^2\\
&\le& q^{d+r-1}\frac{\binom{d}{r}}{d}\binom{d}{d/2}\binom{d/2}{d/4}q^{-d}|E|\\
&=&q^{r-1}\frac{\binom{d}{r}}{d}\binom{d}{d/2}\binom{d/2}{d/4}|E|
\end{eqnarray*}

Clearly, $|I|\le |II|$

It follows  from (\ref{lambda}) that, if $q^{-d}|E|^2(q-1)^r\binom{d}{r}>  q^{r-1}\frac{\binom{d}{r}}{d}\binom{d}{d/2}\binom{d/2}{d/4}|E|$, that is  if 
$$|E|>\frac{q^{d-1}}{d}\binom{d}{d/2}\binom{d/2}{d/4}$$ then 
 $\lambda_r>0.$

\end{proof}


\end{document}